\documentclass[11pt]{article}
\usepackage{fullpage}
\usepackage{amsfonts, amsthm, amssymb, amsmath}
\usepackage{helvet}
\usepackage{framed}
\usepackage{verbatim}
\usepackage{epsfig}
\usepackage[pdftex, pagebackref=true, colorlinks=true, urlcolor=blue, citecolor=blue, linkcolor=blue]{hyperref}

\newtheorem{thm}{Theorem}
\newtheorem{lem}[thm]{Lemma}
\newtheorem{cor}[thm]{Corollary}
\newtheorem{defn}[thm]{Definition}
\newtheorem{clm}[thm]{Claim}
\newtheorem{cons}[thm]{Construction}
\newtheorem{prop}[thm]{Proposition}

\newtheorem{conj}[thm]{Conjecture}
\newtheorem{obs}[thm]{Observation}

\newenvironment{theorem}{\begin{thm}\begin{rm}}%
{\end{rm}\end{thm}}
\newenvironment{lemma}{\begin{lem}\begin{rm}}%
{\end{rm}\end{lem}}
\newenvironment{corollary}{\begin{cor}\begin{rm}}%
{\end{rm}\end{cor}}
\newenvironment{definition}{\begin{defn}\begin{em}}%
{\end{em}\end{defn}}
{\end{rm}\end{clm}}
{\end{em}\end{cons}}
{\end{em}\end{prop}}
{\end{rm}\end{conj}}
\newenvironment{observation}{\begin{obs}\begin{rm}}%
{\end{rm}\end{obs}}

\newcommand{\secref}[1]{\hyperref[#1]{Section \ref{#1}}}
\newcommand{\thref}[1]{\hyperref[#1]{Theorem \ref{#1}}}
\newcommand{\defref}[1]{\hyperref[#1]{Definition \ref{#1}}}
\newcommand{\cororef}[1]{\hyperref[#1]{Corollary \ref{#1}}}
\newcommand{\propref}[1]{\hyperref[#1]{Proposition \ref{#1}}}
\newcommand{\remref}[1]{\hyperref[#1]{Remark \ref{#1}}}
\newcommand{\lemref}[1]{\hyperref[#1]{Lemma \ref{#1}}}
\newcommand{\clref}[1]{\hyperref[#1]{Claim \ref{#1}}}
\newcommand{\consref}[1]{\hyperref[#1]{Construction \ref{#1}}}
\newcommand{\figref}[1]{\hyperref[#1]{Figure \ref{#1}}}
\newcommand{\eqnref}[1]{\hyperref[#1]{Equation \ref{#1}}}
\newcommand{\subroutineref}[1]{\hyperref[#1]{Subroutine \ref{#1}}}
\newcommand{\apref}[1]{\hyperref[#1]{Appendix \ref{#1}}}
\newcommand{\conjref}[1]{\hyperref[#1]{Conjecture  \ref{#1}}}
\newcommand{\obsref}[1]{\hyperref[#1]{Observation \ref{#1}}}

\newcommand{\abs}[1]{\left| #1 \right|} 

\newcommand{\set}[1]{\left\{ #1 \right\}}

\usepackage{tikz}
\usetikzlibrary{arrows}

\def\krel{%
    \ensuremath{%
        \mathrel{%
            \begin{tikzpicture}[baseline=-0.8ex]
                \draw [line width=0.06ex, line join=round, stealth'-stealth']
                    (1.5ex,0.67ex) -- (0, 0) -- (1.5ex, -0.67ex);
                \draw [line width=0.06ex, line cap=round]
                    (0.03ex, -1ex) -- (1.47ex, -1ex);
            \end{tikzpicture}%
        }%
    }%
}

\title{Forbidden Directed Minors and Kelly-width}
\vspace{0.15in}
\author{
Shiva Kintali \footnote{Department of Computer Science, Princeton University, Princeton, NJ 08540. Email : {\em{kintali@cs.princeton.edu}}}\ \ \ \ \ \ 
Qiuyi Zhang \footnote{Mathematics Department, Princeton University, Princeton, NJ 08540. Email : {\em{qiuyiz@princeton.edu}}}
}

\begin{document}
\maketitle

\begin{abstract}

Partial 1-trees are undirected graphs of treewidth at most one. Similarly, partial 1-DAGs are directed graphs of KellyWidth at most two. It is well-known that an undirected graph is a partial 1-tree if and only if it has no $K_3$ minor. In this paper, we generalize this characterization to partial 1-DAGs. We show that partial 1-DAGs are characterized by three forbidden directed minors, $K_3, N_4$ and $M_5$. \\

\end{abstract}

\section{Introduction}
Treewidth and pathwidth (and their associated decompositions) played a crucial role in the development of graph minor theory (\cite{graph-minors-1}$\cdots$\cite{graph-minors-23}) and proved to be algorithmically and structurally important graph parameters. Treewidth (resp. pathwidth) measures how similar a graph is to a tree (resp. path). Treewidth has several equivalent characterizations in terms of elimination orderings, elimination trees (used in symmetric matrix factorization), partial $k$-trees and cops and (visible and eager) robber games. Several problems that are NP-hard on general graphs are solvable in polynomial time (some even in linear time, some are fixed-parameter tractable) on graphs of bounded treewidth by using dynamic programming on a tree-decomposition of the input graph. Similarly pathwidth has several equivalent characterizations in terms of vertex separation number, node searching number, interval thickness (i.e., one less than the maximum clique size in an interval supergraph) and cops and (invisible and eager) robber games. Several problems that are NP-hard on general graphs are efficiently solvable on graphs of bounded pathwidth.

Motivated by the success of treewidth and pathwidth, efforts have been made to generalize these concepts to digraphs. Directed treewidth \cite{dtw-definition}, D-width \cite{dwidth-definition}, DAG-width \cite{dagw-definition1, dagw-definition2, dagw-journal} and Kelly-width \cite{kellywidth-definition} are some such notions which generalize treewidth, whereas directed pathwidth \cite{dpw-definition} generalizes pathwidth. All these parameters have associated decompositions called arboreal decompositions, D-decompositions, DAG-decompositions, Kelly-decompositions and directed path decompositions respectively. Hamitonian cycle, Hamiltonian path, $k$-disjoint paths and weighted disjoint paths are solvable in polynomial time on digraphs of bounded directed treewidth. In addition to these problems, parity games are solvable in polynomial time on digraphs of bounded DAG-width. Directed treewidth and DAG-width started an interesting line of research but they suffer from some disadvantages. Directed treewidth is not monotone under butterfly minors (a natural set of directed minor operations) and they do not have an {\em exact} characterization in terms of cops and robber games. DAG-width is an improvement over these parameters with an exact characterization of cops and (visible and eager) robber games. Unfortunately it is not known whether there exists a {\em linear size} DAG-decomposition for digraphs of bounded DAG-width. The best known upper bound on the size of DAG-decompositions of digraphs of DAG-width $\leq k$ is $O(n^k)$.

Currently, Kelly-width is the best known generalization of treewidth. Hunter and Kreutzer \cite{kellywidth-definition} showed that Kelly-width not only generalizes treewidth but is also characterized by several equivalent notions such as directed elimination orderings, elimination DAGs \cite{elimination-DAGs} (used in asymmetric matrix factorization), partial $k$-DAGs and cops and (invisible and inert) robber games. Hamitonian cycle, Hamiltonian path, $k$-disjoint paths, weighted disjoint paths and parity games are solvable in polynomial time on digraphs of bounded Kelly-width. The size of Kelly-decompositions can be made linear and their structure is suitable for dynamic-programming-type algorithms (see \cite{kellywidth-definition} for more details).

The graph minor theorem \cite{graph-minors-20} (i.e., undirected graphs are well-quasi-ordered under the minor relation) implies that every minor-closed family of undirected graphs has a finite set of minimal forbidden minors. In particular, it implies that for all $k \geq 0$, graphs of treewidth (or pathwidth) $\leq k$ are characterized by a finite set of forbidden minors. The complete sets of forbidden minors are known for small values of treewidth and pathwidth. A graph has treewidth at most one (resp. two) if and only if it is $K_3$-free (resp. $K_4$-free). Graphs of treewidth at most three are characterized by four forbidden minors ($K_5$, the graph of the octahedron, the graph of the pentagonal prism, and the Wagner graph) \cite{forbidden-minors-partial-3-trees, forbidden-minors-partial-3-trees-SatyaTung}. Graphs of pathwidth at most one (resp. two) are characterized by two (resp. 110) forbidden minors \cite{2-minors-pathwidth,110-pathwidth}.

A natural question is ``{\em are partial $k$-DAGs (i.e., digraphs of bounded Kelly-width) characterized by a finite set of forbidden directed minors} ?". Unfortunately there is no generalization of the graph minor theorem for digraphs yet. Existing notions of directed minors (eg. directed topological minors \cite{hunter-thesis}, butterfly minors \cite{dtw-definition}, strong contractions \cite{tournament-wqo}, directed immersions \cite{tournament-immersion-wqo}) do not imply well-quasi-ordering of all digraphs (see \cite{kintali-minors1} for more details).

Recently, the first author \cite{kintali-minors1} introduced a notion of {\em{directed minors}} based on several operations of contracting special subset of directed edges, conjectured that digraphs are well-quasi-ordered under the proposed {\em{directed minor relation}} and proved the conjecture for some special classes of digraphs. This conjecture implies that every family of digraphs closed under the directed minor operations (see \secref{sec:dir-minors}) are characterized by a finite set of minimal {\em forbidden directed minors}. In particular, it implies that for all $k \geq 0$, digraphs of Kelly-width (or DAG-width, or directed pathwidth) $\leq k$ are characterized by a finite set of forbidden directed minors (see \cite{kintali-minors1} for more details).

The current paper is motivated by the question ``{\em what are the complete sets of directed forbidden minors for digraphs with small values of Kelly-width} ?". We exhibit the sets of {\em forbidden directed minors} for digraphs with Kelly-width one and two (i.e., partial 0-DAGs and partial $1$-DAGs). We prove that partial 0-DAGs are characterized by one forbidden directed minor (see \lemref{lem:charac-dags}) and partial 1-DAGs are characterized by three forbidden directed minors (see \figref{fig:kw-minors}).

\subsection{Notation and Preliminaries}

We use standard graph theory notation and terminology (see \cite{diestel-txt}). For a directed graph (digraph) $G$, we write $V(G)$ for its vertex set (or node set) and $E(G)$ for its edge set. All digraphs in this paper are finite and simple (i.e., no self loops and no multiple edges) unless otherwise stated. For an edge $e = (u, v)$, we say that $e$ is an edge {\em from $u$ to $v$}. We say that $u$ is the {\em tail} of $e$ and $v$ is the {\em head} of $e$. We also say that $u$ is an in-neighbor of $v$ and $v$ is an out-neighbor of $u$. The out-neighbors of a vertex $u$ is given by $N_{out}(u) = \set{v : (u, v) \in E}$ and the in-neighbors of a vertex $v$ is given by $N_{in}(v) = \set{u: (u, v) \in E}$. Let $d_{out}(u) = |N_{out}(u)|$ (resp. $d_{in}(v) = |N_{in}(v)|$) denote the out-degree of $u$ (resp. in-degree of $v$).

For $S \subseteq V(G)$ we write $G[S]$ for the subgraph induced by $S$, and $G \setminus S$ for the subgraph induced by $V(G) \setminus S$. For $F \subseteq E(G)$, we write $G[F]$ for the subgraph with vertex set equal to the set of endpoints of $F$, and edge set equal to $F$.

For a digraph $G$, let $\overline{G}$ be the undirected graph, where $V(\overline{G}) = V(G)$ and $E(\overline{G}) = \{\{u,v\} : (u,v) \in E(G)\}$. We say that $\overline{G}$ is the {\em underlying} undirected graph of $G$. For an {\em{undirected}} graph $G$, let $\overset\leftrightarrow{G}$ be the {\it digraph} obtained by replacing each edge $\{u,v\}$ of $G$ by two directed edges $(u,v)$ and $(v,u)$. We say that $\overset\leftrightarrow{G}$ is the {\em bidirected} graph of $G$.

A {\em directed (simple) path} in $G$ is a sequence of vertices $v_1,v_2,\dots,v_l$ such that for all $1 \leq i \leq l-1$, $(v_i,v_{i+1}) \in E(G)$. For a subset $X \subseteq V(G)$, the set of vertices reachable from $X$ is defined as: $Reach_G{(X)} := \{v \in V (G) : \mbox{there is a directed path to $v$ from some}\ u \in X \}$. We say that $G$ is {\em weakly connected} if $\overline{G}$ is connected. We say that $G$ is {\em strongly connected} if, for every pair of vertices $u,v \in V(G)$, there is a directed path from $u$ to $v$ and a directed path from $v$ to $u$.

We use the term DAG when referring to directed acyclic graphs. Let $T$ be a DAG. For two distinct nodes $i$ and $j$ of $T$, we write $i \prec_T j$ if there is a directed walk in $T$ with first node $i$ and last node $j$. For convenience, we write $i \prec j$ whenever $T$ is clear from the context. For nodes $i$ and $j$ of $T$, we write $i \preceq j$ if either $i=j$ or $i \prec j$. For an edge $e=(i,j)$ and a node $k$ of $T$, we write $e \prec k$ if either $j=k$ or $j \prec k$. We write $e \sim i$ (resp. $e \sim j$) to mean that $e$ is incident with $i$ (resp. $j$).

Let $\mathcal{W}=(W_i)_{i \in V(T)}$ be a family of finite sets called {\it node bags}, which associates each node $i$ of $T$ to a node bag $W_i$. We write $W_{\succeq i}$ to denote $\displaystyle\bigcup_{j \succeq i} W_j$. Kelly-width (see \secref{sec:kw}) is based on the following notion of {\it guarding}:

\begin{definition}[{\sc Guarding}]\label{definition:guarding}
Let $G$ be a digraph. Let $W, X \subseteq V(G)$. We say $X$ {\it guards} $W$ if $W \cap X = \emptyset$, and for all $(u,v) \in E(G)$, if $u \in W$ then $v \in W \cup X$.
\end{definition}

In other words, $X$ guards $W$ means that there is no directed path in $G \setminus X$ that starts from $W$ and leaves $W$.

\begin{definition}[{\sc Edge contraction}]\label{definition:contract}
Let $G$ be an undirected graph, and $e = \{u,v\} \in E(G)$. The vertices and edges of the graph $G'$ obtained from $G$ by {\em contracting} $e$ are:
\begin{itemize}
\item{$V(G')=V(G) \setminus {u}$}
\item{$E(G') = (E(G) \cup \{ \{x,v\} : \{x,u\} \in E(G)\})\ \setminus\ \{\{x,u\} :
x \in V (G)\}$ }
\end{itemize}
\end{definition}

\begin{definition}[{\sc Minor}]\label{definition:minor}
Let $G$ and $H$ be undirected graphs. We say that $H$ is a {\em{minor}} of $G$, (denoted by $H \leq G$), if $H$ is isomorphic to a graph obtained from $G$ by a sequence of vertex deletions, edge deletions and edge contractions. These operations may be applied to $G$ in any order, to obtain its directed minor $H$.
\end{definition}

\begin{theorem}
(Robertson-Seymour theorem \cite{graph-minors-20}) Undirected graphs are well-quasi-ordered by the minor relation $\leq$.
\end{theorem}

\begin{corollary}
Every minor-closed family of undirected graphs has a finite set of minimal forbidden minors. In particular, for all $k \geq 0$, graphs of treewidth (or pathwidth) $\leq k$ are characterized by a finite set of forbidden directed minors.
\end{corollary}

\section{Directed Minors}\label{sec:dir-minors}

In this section, we present a subset of the {\em directed minor operations} from \cite{kintali-minors1}. The {\em directed minor relation} in \cite{kintali-minors1} has more operations called source/sink contractions. These operations are not necessary for our results in the current paper.

When we perform the following operations on a digraph $G$ to obtain a digraph $G'$, we remove any resulting self-loops and multi-edges from $G'$. Cycle contraction is a generalization of the edge contraction from \defref{definition:contract}.

\begin{definition}[{\sc Cycle contraction}]\label{definition:cycle-contract}
Let $G$ be a graph, and $C = \{v_1, v_2, \dots, v_l\} \subseteq V(G)$ be a directed cycle in $G$ i.e., $(v_i, v_{i+1}) \in E(G)$ for $1 \leq i \leq l-1$ and $(v_l, v_1) \in E(G)$. The vertices and edges of the graph $G'$ obtained from $G$ by {\em contracting} $C$ are:
\begin{itemize}
\item{$V(G')=\{V(G) \setminus {C} \} \cup \{ w \}$, where $w$ is a new vertex.}
\item{$E(G') = \{ E(G) \setminus \{ \{ (x,v_i) : 1 \leq i \leq l, x \in V(G) \} \cup \{ (v_i,x) : 1 \leq i \leq l, x \in V(G) \} \} \} \\
\cup\ \{ \{ (x,w) : (x,v_i) \in E(G)\ \mbox{for some}\ x \in V(G) \setminus C, 1 \leq i \leq l\} \cup \{ (w,x) : (v_i,x) \in E(G)\ \mbox{for some}\ x \in V(G) \setminus C, 1 \leq i \leq l \} \}$ }
\end{itemize}
\end{definition}

Butterfly contractions (defined by Johnson et al \cite{dtw-definition}) allow us to contract a directed edge $e = (u,v)$ if either $e$ is the only edge with head $v$, or it is the only edge with tail $u$, or both. The following operations, {\em out-contraction} and {\em in-contraction}, are slightly general and allows {\em any} edge to be out-contracted or in-contracted after removing certain incident edges. Out-contracting an edge $(u,v)$ is equivalent to removing all the out-going edges of $u$ and identifying $u$ and $v$. Similarly, in-contracting an edge $(u,v)$ is equivalent to removing all the in-coming edges of $v$ and identifying $u$ and $v$. Hence, we can out-contract or in-contract {\em any} edge of $G$ without creating new paths in $G$.

\begin{definition}[{\sc Out contraction}]\label{definition:out-contract}
Let $G$ be a graph, and $e = (u,v) \in E(G)$. The vertices and edges of the graph $G'$ obtained from $G$ by {\em out-contracting} $e$ are:
\begin{itemize}
\item{$V(G')=V(G) \setminus {u}$}
\item{$E(G') = (E(G) \cup \{ (x,v) : (x,u) \in E(G)\} )\ \setminus\ \{(x,u),(u,x) :
x \in V (G)\}$ }
\end{itemize}
Note that we delete the vertex $u$, the tail of the edge $e = (u,v)$. The vertex $v$ exists in both $G$ and $G'$.
\end{definition}

\begin{definition}[{\sc In contraction}]\label{definition:in-contract}
Let $G$ be a graph, and $e = (u,v) \in E(G)$. The vertices and edges of the graph $G'$ obtained from $G$ by {\em in-contracting} $e$ are:
\begin{itemize}
\item{$V(G')=V(G) \setminus {v}$}
\item{$E(G') = (E(G) \cup \{ (u,x) : (v,x) \in E(G)\} )\ \setminus\ \{(x,v),(v,x) :
x \in V (G)\}$ }
\end{itemize}
Note that we delete the vertex $v$, the head of the edge $e = (u,v)$. The vertex $u$ exists in both $G$ and $G'$.
\end{definition}

\begin{definition}[{\sc Directed minor}]\label{definition:directed-minor}
Let $G$ and $H$ be digraphs. We say that $H$ is a {\em{directed minor}} of $G$, (denoted by $H \krel G$), if $H$ is isomorphic to a graph obtained from $G$ by a sequence of vertex deletions, edge deletions, cycle contractions, out/in contractions. These operations may be applied to $G$ in any order, to obtain its directed minor $H$.
\end{definition}

\section{Kelly-width and partial $k$-DAGs}\label{sec:kw}

We now present the definitions of Kelly-decomposition, Kelly-width and its equivalent characterizations such as partial $k$-DAGs, directed elimination orderings and cops and (invisible and inert) robber games from \cite{kellywidth-definition}. The class of $k$-trees are generalizations of trees. Similarly, $k$-DAGs are generalizations of DAGs. Directed vertex elimination involves removing vertices from a digraph and adding new directed edges to preserve directed reachability.

\begin{definition}[Kelly-decomposition and Kelly-width \cite{kellywidth-definition}]\label{definition:Kelly-decomposition}
A {\it Kelly-decomposition} of a digraph $G$ is a triple $\mathcal{D} = (T, \mathcal{W}, \mathcal{X})$ where $T$ is a DAG, and $\mathcal{W}=(W_i)_{i\in V(T)}$ and $\mathcal{X}=(X_i)_{i\in V(T)}$ are families of subsets (node bags) of $V(G)$, such that:
\begin{itemize}
\item $\mathcal{W}$ is a partition of $V(G)$. \hfill{\rm{\sf (KW-1)}}
\item For all nodes $i \in V(T), X_i$ guards $W_{\succeq i}$. \hfill{\rm{\sf (KW-2)}}
\item For each node $i \in V(T)$, the children of $i$ can be enumerated as $j_1, ... , j_s$ so that for each $j_q$, $X_{j_q} \subseteq W_i \cup X_i \cup \bigcup_{p<q} W_{\succeq j_p}$. Also, the roots of $T$ can be enumerated as $r_1, r_2, ...$ such that for each root $r_q$, $W_{r_q} \subseteq \bigcup_{p<q} W_{\succeq r_p}$. \hfill{\rm{\sf (KW-3)}}
\end{itemize}
The width of a Kelly-decomposition $\mathcal{D}=(T,\mathcal{W},\mathcal{X})$ is defined as $\max\{|W_i \cup X_i| : i \in V(T)\}$. The {\it Kelly-width} of $G$, denoted by $kw(G)$, is the minimum width over all possible Kelly-decompositions of $G$.
\end{definition}

\begin{definition}[Partial $k$-DAG \cite{kellywidth-definition}]\label{definition:partial-k-DAG}
The class of $k$-DAGs is defined recursively as follows:
\begin{itemize}
\item{A complete digraph with $k$ vertices is a $k$-DAG.}
\item{A $k$-DAG with $n+1$ vertices can be constructed from a $k$-DAG $H$ with $n$ vertices by adding a vertex $v$ and edges satisfying the following:
    \begin{itemize}
    \item{At most $k$ edges from $v$ to $H$ are added}
    \item{If $X$ is the set of endpoints of the edges added in the previous step, an edge from $u \in V(H)$ to $v$ is added if $(u,w) \in E(H)$ for all $w \in X \setminus \{u\}$. Note that if $X = \emptyset$, this condition is true for all $u \in V(H)$.}
    \end{itemize}
A {\em partial $k$-DAG} is a subgraph of a $k$-DAG.
}
\end{itemize}
\end{definition}

\begin{definition}[Directed elimination ordering \cite{kellywidth-definition}]\label{definition:elimination-ordering}
Let $G$ be a digraph.
\begin{itemize}
\item{A directed elimination ordering $\vartriangleleft$ is a linear ordering on $V(G)$.}
\item{Given an elimination ordering $\vartriangleleft := (v_0, v_1, \dots, v_{n-1})$ of $G$, define :
    \begin{itemize}
    \item{$G_0^{\vartriangleleft} := G$ and}
    \item{$G_{i+1}^{\vartriangleleft} := G$ is obtained from $G_{i}^{\vartriangleleft} := G$ by deleting $v_i$ and (if necessary) adding new edges $(u,v)$ if $(u,v_i), (v_i,v) \in E(G_{i}^{\vartriangleleft})$ and $u \neq v$.}
    \end{itemize}
$G_{i}^{\vartriangleleft}$ is the {\em directed elimination graph at step i according to $\vartriangleleft$.}
}
\item{The {\em width} of an elimination ordering is the maximum over all $i$ of the out-degree of $v_i$ in $G_{i}^{\vartriangleleft}$.}
\end{itemize}
\end{definition}

\noindent {\bf Cops and (invisible and inert/lazy) robber games on digraphs} : There are $k$ cops trying to catch a robber on a digraph. The robber occupies a vertex. Each cop either occupies a vertex or moves around in an helicopter. The robber can always see the cops' locations and the helicopters landing. He can move to another vertex at an infinite speed along a cop-free directed path. The goal of the cops is to capture the robber by landing on the vertex currently occupied by him. The goal of the robber is to avoid capture. Hunter and Kreutzer \cite{kellywidth-definition} showed that Kelly-width is characterized by cops and robber game in which the robber is invisible and inert (i.e., the robber may only move if a cop is about to occupy the robber's current vertex).

\begin{theorem}\label{thm:kw-charac}
(Hunter and Kreutzer \cite{kellywidth-definition}) Let $G$ be a digraph. The following are equivalent: (i) $G$ has Kelly-width $\leq k+1$, (ii) $G$ is a partial $k$-DAG and (iii) $G$ has a directed elimination ordering of width $\leq k$ (iv) $k+1$ cops can capture an invisible and inert robber on $G$.
\end{theorem}

\begin{lemma}\label{lem:kw-monotone}
Let $H \krel G$ and $kw(G) \leq k$. Then, $kw(H) \leq k$. In other words, Kelly-width is monotone under the directed minor operations mentioned in \secref{sec:dir-minors}.
\end{lemma}

\begin{proof}
Note that the directed minor operations (from \secref{sec:dir-minors}) do not create new directed cycles i.e., if $H \krel G$ and $H$ contains a directed cycle (on vertices $u$ and $v$) then $G$ also contains a directed cycle (on vertices $u$ and $v$). Hence, these operations  do not help the robber in the above mentioned game i.e., if $k$ cops are sufficient to capture a robber on $G$ then they are sufficient on any directed minor of $G$. Since Kelly-width is precisely characterized by this game (see \thref{thm:kw-charac}), digraphs of bounded Kelly-width are monotone under $\krel$. 
\end{proof}


The above lemma combined with the well-quasi-ordering conjecture of \cite{kintali-minors1} implies that digraphs of bounded Kelly-width are characterized by finite number of forbidden directed minors. The following lemma characterizes digraphs of Kelly-width one (i.e., partial 0-DAGs) in terms of forbidden directed minors. An undirected graph $G$ is a partial 0-tree (i.e., has treewidth zero) if and only if it contains no $K_2$ minor. Similarly, the following lemma states that a digraph $G$ is a partial 0-DAG if and only if it contains no $K_2$ minor.

\begin{lemma} \label{lem:charac-dags}
A digraph $G$ is a partial 0-DAG if and only if it contains no $K_2$ minor.
\end{lemma}

\begin{proof}
From \defref{definition:partial-k-DAG} it is easy to see that partial 0-DAGs are precisely digraphs whose weakly connected components are DAGs. We may assume that $G$ is weakly connected. $K_2$ is the directed graph on two vertices, say $u$ and $v$, with two edges $(u,v)$ and $(v,u)$. If $G$ has a directed cycle with edge set $C \subseteq E(G)$, let $G' = G[C]$. By repeatedly out-contracting the edges of $G'$ we obtain a $K_2$ minor.

To show the other direction, let $G$ be a digraph with a $K_2$ minor. As observed earlier, the directed minor operations do not create new directed cycles i.e., if $H \krel G$ and $H$ contains a directed cycle (on vertices $u$ and $v$) then $G$ also contains a directed cycle (on vertices $u$ and $v$). Hence, there is a directed path from $u$ to $v$ and a directed path from $v$ to $u$ in $G$, implying that $G$ is not a DAG.
\end{proof}

\section{Characterization of  partial 1-DAGs}

It is well-known that an undirected graph is a partial 1-tree (i.e., a forest) if and only if it has no $K_3$ minor. In this section, we generalize this characterization to partial 1-DAGs. We show that partial 1-DAGs are characterized by three forbidden directed minors, $K_3, N_4$ and $M_5$ (see \figref{fig:kw-minors})\footnote{Note that the bidirected edges of $N_4$ and $M_5$ resemble the letters $N$ and $M$ respectively.}. First, we show that any digraph, in which every vertex has out-degree at least 2, contains $K_3, N_4$ or $M_5$ as a directed minor.

\begin{figure}
    \centering
    \includegraphics[height = 2in]{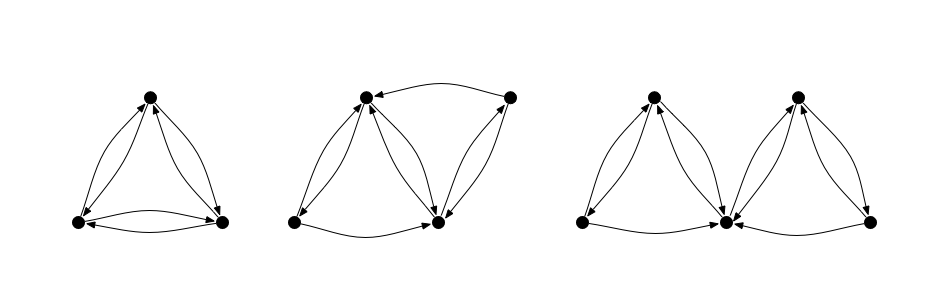}
    \caption{The forbidden directed minors of partial 1-DAGs: $K_3$, $N_4$ and $M_5$}
    \label{fig:kw-minors}
\end{figure}

\begin{theorem}\label{thm:out-two-minors}
Let $G$ be a simple digraph such that every vertex in $V(G)$ has out-degree $\geq 2$. Then $G$ contains $K_3, N_4$ or $M_5$ as a directed minor.
\end{theorem}

\begin{proof}
Let $G$ be a {\em minor-minimal} counter-example to the above theorem. Let $V$ and $E$ denote the vertex set and the edge set of $G$ respectively. We may assume that $G$ is weakly connected, or else any weakly-connected component of $G$ is a smaller counter-example, contradicting the minimality of $G$. We may assume that every vertex of $G$ has outdegree {\em exactly} 2, or else we may delete an edge and obtain a smaller counter-example. This contradicts the minimality of $G$.

Furthermore, we may assume $G$ is strongly connected, or else we may partition $V$ into two non-empty sets $A$ and $B$ such that there are no edges {\em from} any vertex in $A$ {\em to} any vertex in $B$. Now, consider $G[A]$, the induced subgraph on $A$. Every vertex in $G[A]$ has outdegree 2 and $G[A]$ does not contain $K_3, N_4$ or $M_5$ as a directed minor (because $G$ does not contain $K_3, N_4$ or $M_5$ as a directed minor and $G[A]$ is a subgraph of $G$). This contradicts the minimality of $G$.

We may assume that there is an edge $(u, v) \in E$ such that $(v, u) \not \in E$, since otherwise $G$ has a $K_3$ minor. We claim that there exists a vertex in $G$ whose out-neighbors are $u$ and $v$. To see this, consider the graph $G'$ obtained by {\em out-contracting} the edge $(u,v)$. Note that the out-degree of $v$ in $G'$ is still two. By the minimality of $G$, there must exist a vertex $w \neq v$ in $G'$ whose out-degree is 1. But the out-degree of $w$ in $G$ is 2, implying that the two out-neighbors of $w$ in $G$ are $u$ and $v$. Hence, for every edge $(u, v) \in E$ such that $(v, u) \not \in E$ there exists a vertex $w$ such $(w, u), (w, v) \in E$. We say that $w$ {\em blocks} the edge $(u,v)$. Note that there can be more than one vertex {\em blocking} an edge $(u,v)$. On the other hand, a vertex $w$ can {\em block} at most one edge because every vertex in $G$ has out-degree two.

Let $\{u,v\}$ be a {\em bidirected edge} in $G$ i.e., $(u, v), (v, u) \in E$. We claim that there exists a vertex $w$ in $G$ such that $(w, u), (w, v) \in E$ or $(u, w), (v, w) \in E$. To see this, consider the graph $G'$ obtained by {\em contracting} the bidirected edge $\{u,v\}$ into a new vertex $x$. Note that $G'$ is strongly-connected. By the minimality of $G$, there must exist a vertex $w$ in $G'$ whose out-degree is 1. If $w \neq x$, then $(w, u), (w, v) \in E$. Else if $w = x$, then $u$ and $v$ have a common out-neighbor (say $w'$) in $G$ i.e., $(u, w'), (v, w') \in E$. \\

Now, let $\alpha$ be the number of {\em directed} edges of $G$ (i.e. edges $(u, v) \in E$ such that $(v, u) \not \in E$) and $\beta$ be the number of {\em bidirected} edges (i.e., pairs $\{u,v\}$ such that $(u, v), (v, u) \in E$). Since every vertex of $G$ has outdegree 2, we have $\alpha + 2\beta = 2\abs{V}$. As shown earlier, any directed edge must have at least one vertex {\em blocking} it. Any {\em blocker vertex} blocks a unique directed edge. Hence $\alpha \leq \abs{V}$. This implies that $\beta \geq \abs{V}/2$. In particular, there must exist at least one bidirected edge in $G$, say $ e= (u,v)$. Using one of the previous arguments, there exists $w \in V$ such that $(w, u), (w, v) \in E$ or $(u, w), (v, w) \in E$. Now, we consider several cases and show that in every case $G$ contains a $K_3$, $N_4$ or $M_5$ minor. \\

\textbf{Case 1}: Let $(w, u), (w, v), (u, w), (v, w) \in E$ implying that $\{u, v, w\}$ induce $K_3$. So, $G$ contains a $K_3$ minor, a contradiction. \\

\textbf{Case 2}: Let $(w, u), (w, v), (u, w) \in E, (v, w) \not \in E$\footnote{the case when $(w, u), (w, v), (v, w) \in E, (u, w) \not \in E$ is similar}. Since the out-degree of $v$ is two, there must exist a vertex $a$ such that $a \neq u, a \neq w$ and $(v, a) \in E$. By strong-connectivity of $G$, there is a directed path $P_{av_n} = av_1v_2...v_n$ such that $v_n = u$ (or $v$ or $w$) and $v_i \not \in \set{u, v, w}$ for $1 < i < n$. \\

\textbf{Case 2.1}: Let $v_n = w$ i.e., we have a path from $a$ to $w$. Now we contract $P_{aw}$ by applying the out-contract operation in a backwards fashion. We first out-contract the edge $(v_{n-1}, v_n)$ i.e., we remove all the out-going edges of $v_{n-1}$ and identify $v_{n-1}$ and $v_n$ such that the remaining vertex is $v_n$ (see \defref{definition:out-contract}). Now we out-contract $(v_{n-2}, v_n)$ and so on until $a$ and $v_n$ are identified. After this process we are left with an edge $(v, v_n) = (v, w)$. Note that the edges between $u, v$ and $w$ are not affected by this process. Therefore, we get a $K_3$ minor on the vertices $\set{u, v, w}$.\\

\textbf{Case 2.2}: Let $v_n = u$ i.e., we have a path from $a$ to $u$. Now we have two cases based on whether there is an edge from $a$ to $v$. \\

\textbf{Case 2.2.1}: $(a, v) \in E$ i.e., $\{a,v\}$ is a bidirected edge. We apply the out-contract operation on $P_{au}$ in a backwards fashion until we obtain the edge $(a, v_n) = (a, u)$. Now observe that we have an $N_4$ minor on the vertices $\set{u, v, w, a}$.  \\

\textbf{Case 2.2.2}: $(a, v) \not \in E$ i.e., there exists a vertex $b \in V$ that blocks the edge $(v,a)$ i.e., $(b, v), (b, a) \in E$. Note that $b \neq u$ and $b \neq w$ because $u$ and $w$ already have out-degree two. Since $G$ is strongly connected, there is a path $P_{ab}$ from $a$ to $b$. Consider the induced subgraph $H = G[\{u,v,w\}]$. The two out-neighbors of $u$ and $w$ are inside $H$. The only out-going edge from $H$ to $G[E(G) \setminus E(H)]$ is $(v,a)$. Hence, $P_{ab}$ being a simple path from $a$ to $b$, does not intersect $\set{u, v, w}$. By its definition $P_{au}$ does not intersect $\set{v, w}$. Also, $P_{au}$ does not intersect $b$ because the two out-neighbors of $b$ are $v$ and $a$.

Now we contract $P_{ab}$ and $P_{au}$ carefully. Let $x$ be the last vertex along the path $P_{au}$ that is also on $P_{ab}$. Let $P_{ax}$ be the subpath of $P_{ab}$. We out-contract the edges of $P_{ax}$ in the forwards direction until the entire path is identified into a single vertex $x$. Now $P_{xb}$ and $P_{xu}$ are internally vertex-disjoint paths. We out-contract the edges of $P_{xb}$ and $P_{xu}$ in a backwards fashion until we have edges $(x, b), (x, u)$. Now, we out-contract $(b, v)$ to obtain the edge $(a, v)$. Now observe that we have an $N_4$ minor on the vertices $\set{u, v, w, a}$. \\

\textbf{Case 2.3}: Let $v_n = v$ i.e., we have a path from $a$ to $v$. Now we have two cases based on whether there is an edge from $a$ to $v$. \\

\textbf{Case 2.3.1}: $(a, v) \in E$ i.e., $\{a,v\}$ is a bidirected edge. Hence, there exists a vertex $b$ such that $(b, v), (b, a) \in E$ or $(v, b), (a, b) \in E$. \\

\textbf{Case 2.3.1.1}: If $(b, v), (b, a) \in E$ then $b \neq u$ and $b \neq w$ because $u$ and $w$ already have out-degree two. Since $G$ is strongly connected, there must be a path $P_{ab}$ from $a$ to $b$. Using the argument from Case 2.2.2, $P_{ab}$ does not intersect $\set{u, v, w}$. Out-contracting $P_{ab}$ backwards until we obtain the edge $(a, b)$ gives an $M_5$ minor on $\set{u, v, w, a, b}$. \\

\textbf{Case 2.3.1.2}: If $(v, b), (a, b) \in E$, then $b = u$ because $v$'s out-degree is two. This gives an $N_4$ minor on $\set{u, v, w, a}$. \\

\textbf{Case 2.3.2}: $(a, v) \not \in E$ i.e., there exists a vertex $b \in V$ that blocks the edge $(v,a)$ i.e., $(b, v), (b, a) \in E$. Note that $b$ is distinct from $\{u, v, w, a\}$. Also $(v, b) \not \in E$ because $v$'s out-degree is two. So, there exists a vertex $c \in V$ that blocks the edge $(b,v)$ i.e., $(c,b), (v,c) \in E$ and $c$ is distinct from $\{u, v, w, b\}$. Also, if $c = a$, then we get an $M_5$ minor on $\{u,v,w,a,b\}$. So, we may assume that $c \neq a$. Now we have two cases based on whether there is an edge from $a$ to $b$.\\

\textbf{Case 2.3.2.1}: $(a, b) \in E$. Since $G$ is strongly connected, there is a path $P_{ac}$ from $a$ to $c$ such that $P_{ac}$ does not intersect $\set{u, v, w}$. If $b$ is not on the path $P_{ac}$ then we out-contract $P_{ac}$ backwards till $a$ is identified with $c$. Now we contract $(c, v)$ to obtain the edge $(a, v)$. We get an $M_5$ minor on $\set{u, w, v, a, b}$. If $b$ is on the path $P_{ac}$, let $P_{bc}$ be the subpath of $P_{ac}$. Note that $P_{bc}$ does not intersect $\set{u, v, w, a}$. We out-contract $P_{bc}$ backwards until we get the edge $(b, c)$. Now we out-contract $(a, b)$ to obtain the edge $(v, b)$. We get an $M_5$ on $\set{u, v, w, b, c}$. \\

\textbf{Case 2.3.2.2}: $(a, b) \not \in E$. Then there exists a vertex $d \in V$ that blocks the edge $(a, b)$ i.e., $(d, a), (d, b) \in E$. Note that $d \neq c$ because $c$ blocks $(b,v)$ and $d$ blocks $(b,a)$. Since $G$ is strongly connected there are paths $P_{ad}$ and $P_{ac}$. By arguments similar to Case 2.2.2,  we conclude that $P_{ad}$ does not intersect $\set{u, v, w, b, c}$ and $P_{ac}$ does not intersect $\set{u, v, w, b, d}$.  We out-contract the paths $P_{ad}$ and $P_{ac}$ backwards until we get the edges $(a,d)$ and $(a, c)$. Now we out-contract the edges $(c, v)$ and $(d, b)$ to obtain the edges $(a, b)$ and $(a, v)$. This results in an $M_5$ minor on $\set{u, v, w, a, b}$. \\

\textbf{Case 3}: $(w, u), (w, v) \in E$ and $(u, w), (v, w) \not \in E$. \\

Since $u$ and $v$ must have two out-neighbors, there exists $a, b \neq w$ such that $(v, a), (u, b) \in E$. If $a = b$, then by strong-connectivity of $G$, there exists a path $P_{aw}$ from $a$ to $w$ that does not intersect $u$ or $v$. Now we out-contract $P_{aw}$ backwards until $a$ and $w$ are identified to obtain edges $(u, w)$ and $(v, w)$. This gives a $K_3$ minor on $\{u,v,w\}$. So, we may assume $a \neq b$. 

By strong connectivity of $G$, there exist directed paths from $a$ to $w$ (say $P_{aw}$) and $b$ to $w$ (say $P_{bw}$). Now we claim that we can choose these paths such that either $P_{aw} \cap {\{u,v\}} = \emptyset$ or $P_{bw} \cap {\{u,v\}} = \emptyset$. If $P_{bw}$ contains any of $\{u, v\}$, then note that since $\set{u, v}$ only has out-neighbors $\set{a, b}$, we conclude that $P_{bw}$ must contain $a$. Then, the subpath of $P_{bw}$ from $a$ to $w$ (say $P_{aw}'$) must not intersect $\set{u, v}$. because $P_{aw}$ and $P_{bw}$ are simple paths. This gives us $P_{aw}'$, a simple path from $a$ to $w$ that does not intersect $\set{u, v}$. \\

\textbf{Case 3.1}: If there are paths $P_{aw}, P_{bw}$ such that they both do not intersect $\set{u, v}$, then let $x$ denote the first vertex along $P_{aw}$ that is also on $P_{bw}$. Note that $P_{ax} $ and $P_{bx}$ are now internally vertex-disjoint. Now, we out-contract $P_{ax}$ and $P_{bx}$ backwards until we identify $x$, $a$ and $b$. We also out-contract $P_{xw}$ backwards to identify $a$, $b$ and $w$, which gives a $K_3$ minor on $\set{u, v, w}$. \\

\textbf{Case 3.2}: Without loss of generality, we may assume that $P_{aw}$ does not intersect $\set{u, v}$ and all paths from $b$ to $w$ intersects either $u$ or $v$. Note that all paths starting from $b$ must first intersect $\set{u, v}$ before any vertex on the path $P_{aw}$, or else we can find a path from $b$ to $w$ that follows along $P_{aw}$ without intersecting $\set{u, v}$. \\

\textbf{Case 3.2.1}: There exists a path $P_{bu}$ from $b$ to $u$ such that $P_{bu}$ does not contain $v$. We apply the contraction steps from Case 2.2. These contractions do not affect $P_{aw}$. Now we out-contract $P_{aw}$ backwards until $a$ and $w$ are identified to obtain the edge $(u, w)$ resulting in an $N_4$ minor. \\

\textbf{Case 3.2.2}: There exists a path $P_{bv}$ from $b$ to $v$. We may assume that $P_{bv}$ does not contain $u$. We apply the contraction steps from Case 2.3. These contractions do not affect $P_{aw}$. Now we out-contract $P_{aw}$ backwards until $a$ and $w$ are identified to obtain the edge $(u, w)$ resulting in an $M_5$ minor. \\

\textbf{Case 4}: $(w, u), (w, v) \not \in E, (u, w), (v, w)  \in E$. We may assume that none of the cases above hold i.e., for every bidirected edge in $G$, we are in Case 4. We say that $(u, w)$ and $(v, w)$ are the ``blocker edges" of $w$ for the edge $e = (u,v)$. These blocker edges cannot be shared by two different bidirected edges because of the out-degree of $u$ and $v$ is two.

Hence, $\alpha \geq 2\beta$, where $\alpha$ is the number of directed edges and $\beta$ is the number of bidirected edges. As mentioned earlier, we have $\abs{V} \geq \alpha \geq 2\beta$ and $\alpha + 2\beta = 2\abs{V}$. This implies $\abs{V} = \alpha = 2\beta$. The equality $\abs{V} = \alpha$ implies that every vertex must block some directed edge. In particular, $w$ must block some directed edge $(a, b) \in E$ implying $(b, a) \notin E$. The equality of $\alpha = 2\beta$ forces all directed edges to be blocker edges of some bidirected edge. \\

\textbf{Case 4.1}: $(w, a) \in E$ and $(a, w) \not \in E$. Then, $(w, a)$ must be the blocker edge for some bidirected edge $e'$. This implies that $e'$ must be incident to $w$ and since $w$ has outdegree 2, we conclude that $e' = (w,b)$ and $(b, a) \in E$ which is a contradiction (recall that $(b, a) \notin E$). \\

\textbf{Case 4.2}: $(w, a), (a, w) \in E$. Since $G$ is strongly connected, there is a path from $b$ to $\set{u, v}$. Without loss of generality, there is a path from $b$ to $v$ such that $P_{bv}$ does not contain $u$. Also, $P_{bv}$ does not contain $\set{w, a}$. Now we out-contract $P_{bv}$ until $b$ and $v$ are identified to obtain the edges $(w, v), (a, v)$ resulting in an $N_4$ minor on $\set{u, v, w, a}$. \\

We conclude that no such counter-example $G$ exists and hence the theorem is true.
\end{proof}

\vspace{0.1in}

Now we are ready to prove our main theorem. Note that the directed elimination ordering (see \defref{definition:elimination-ordering}) defines an order of ``eliminating" the vertices of $G_i$ to obtain $G_{i+1}$. This is done by deleting a vertex (say $v$) and adding edges from all in-neighbors of $v$ to all out-neighbors of $v$, without introducing loops or multiple edges. \thref{thm:kw-charac} implies that a digraph $G$ has Kelly-width at most $k + 1$ if and only if $G$ can be reduced to the null digraph by repeatedly eliminating a vertex of out-degree at most $k$. The following observations are very crucial.

\begin{observation}
Let $G$ be a digraph and $u \in V(G)$ be a vertex of out-degree zero. Then, eliminating $u$ is equivalent to deleting $u$.
\end{observation}

\begin{observation}
Let $G$ be a digraph and $u \in V(G)$ be a vertex of out-degree one. Let $v$ be the out-neighbor of $u$. Then, eliminating $u$ is equivalent to out-contracting the edge $(u,v)$.
\end{observation}

\begin{observation}\label{obs-minor}
If $H$ is obtained from $G$ by eliminating a vertex of out-degree $\leq 1$ then $H \krel G$.
\end{observation}

\begin{theorem} 
A digraph $G$ is a partial 1-DAG if and only if it contains no $K_3, N_4$ or $M_5$ as a directed minor.  
\end{theorem}

\begin{proof}
Using the directed elimination ordering, it is easy to see that $K_3, N_4$ and $M_5$ are partial 2-DAGs (but not partial 1-DAGs) and all their proper directed minors are partial 1-DAGs. Recall that \lemref{lem:kw-monotone} states that partial $k$-DAGs are closed under $\krel$. Hence partial 1-DAGs cannot contain $K_3, N_4$ or $M_5$ minors.

To prove the other direction, let $G$ be a partial $k$-DAG, for some $k \geq 2$, that does not contain $K_3, N_4$ or $M_5$ minors. Now we repeatedly eliminate vertices of out-degree $\leq 1$ until there are no such vertices. Let $H$ be the resulting graph. By \obsref{obs-minor}, $H \krel G$. If $H$ is the null digraph, then $G$ is a partial 1-DAG, a contradiction. So, $H$ is a non-null digraph in which all vertices have out-degree $\geq 2$. By \thref{thm:out-two-minors}, $H$ must contain a $K_3, N_4$ or $M_5$ as a minor, a contradiction.
\end{proof}

\section{Conclusion and Open Problems}

In this paper, we proved that partial 1-DAGs are characterized by three forbidden directed minors, $K_3, N_4$ and $M_5$. Our result generalizes the forbidden minor characterization of partial 1-trees. As mentioned in the introduction, the complete sets of forbidden minors are also known for graphs of treewidth at most two and three. An interesting open problem is to show such forbidden directed minor characterizations of partial 2-DAGs and partial 3-DAGs by extending the techniques used in the current paper. Graphs of pathwidth at most one are characterized by two forbidden minors \cite{2-minors-pathwidth}. In a sequel to this paper, we exhibit the complete set of forbidden directed minors of digraphs with directed pathwidth at most one \cite{kintali-zhang-dpw-minors}. As mentioned earlier, a suitable {\em directed graph minor theorem}, implies that for all $k \geq 0$, digraphs of Kelly-width (or DAG-width, or directed pathwidth) $\leq k$ are characterized by a finite set of forbidden directed minors (see \cite{kintali-minors1} for more details).

\bibliographystyle{alpha}
\bibliography{../../bib-kintali,../../bib-twbook}

\end{document}